\documentclass[11pt,twoside]{article}

\newcommand{\HeadTitle}{Right edge rates of the zeros of $\widetilde{\Xi}_n$ and $\widetilde{\Lambda}_n$}

\newcommand{\HeadTitleTwo}{\begin{center}
\Large{\textit{Right edge rates of the zeros of $\widetilde{\Xi}_n$ and $\widetilde{\Lambda}_n$}}
\end{center}}

\usepackage[margin=0.8in]{geometry}
\usepackage{amsmath, amssymb, amsfonts, mathtools}
\usepackage{amsthm}
\usepackage{mathrsfs}
\usepackage{stmaryrd}
\usepackage{esint}
\usepackage{eqnalign}

\usepackage{graphicx}
\usepackage{enumitem}
\usepackage[most]{tcolorbox}
\usepackage{float}
\usepackage{listings}
\usepackage{array}
\usepackage{booktabs}
\usepackage{xcolor}
\usepackage{emptypage}
\usepackage{tensor}
\usepackage{tabularx}

\usepackage{fancyhdr}
\usepackage{hyperref}
\usepackage[nameinlink,noabbrev]{cleveref}

\usepackage{titlesec}
\usepackage[T1]{fontenc}
\usepackage[utf8]{inputenc}
\usepackage{lmodern}
\usepackage{titling}

\usepackage[backend=biber,style=numeric]{biblatex}
\providecommand{\HeadTitle}{} 
\providecommand{\HeadTitleTwo}{} 
\providecommand{\HeadAuthor}{Luc Ramsès TALLA WAFFO} 

\titleformat{\section}[block]
  {\normalfont\large\bfseries\itshape\centering}
  {§\thesection.}
  {1em}
  {}

\titleformat{\subsection}
  {\normalfont\normalsize\bfseries}
  {\thesubsection}
  {1em}
  {}

\newtheorem{theorem}{Theorem}[section]
\newtheorem{lemma}[theorem]{Lemma}

\crefname{example}{example}{examples}
\Crefname{example}{Example}{Examples}

\crefname{corollary}{corollary}{corollaries}
\Crefname{corollary}{Corollary}{Corollaries}

\crefname{definition}{definition}{definitions}
\Crefname{definition}{Definition}{Definitions}

\crefname{remark}{remark}{remarks}
\Crefname{remark}{Remark}{Remarks}

\crefname{conjecture}{conjecture}{conjectures}
\Crefname{conjecture}{Conjecture}{Conjectures}

\crefname{lemma}{lemma}{lemmas}
\Crefname{lemma}{Lemma}{Lemmas}

\crefname{proposition}{proposition}{propositions}
\Crefname{proposition}{Proposition}{Propositions}

\crefname{theorem}{theorem}{theorems}
\Crefname{theorem}{Theorem}{Theorems}

\numberwithin{equation}{section}

\usepackage{fontspec}
\begin{document}

\thispagestyle{fancy}

\vspace{0.2cm}

\begin{center}
\Large{\HeadTitleTwo}
\end{center}

\hspace{3cm}

\begin{center}
Luc Ramsès TALLA WAFFO \\
Technische Universität Darmstadt\\
Karolinenplatz 5, 64289 Darmstadt, Germany\\
ramses.talla@stud.tu-darmstadt.de\\
\vspace{0.5cm}
April 28, 2026
\end{center}

\begin{abstract}
We consider the two families of even polynomials $\Xi_n$ and $\Lambda_n$ studied in~\cite{TallaWaffo2026arxiv2602.16761}, together with the rescaled polynomials $\widetilde{\Xi}_n(x):=\Xi_n(\sqrt{x})$ and $\widetilde{\Lambda}_n(x):=\Lambda_n(\sqrt{x})$, $n\ge2$. Their zeros are real, simple, and contained in $(0,1)$. Writing them as $0<x^{(\Xi)}_{1,n}<\cdots<x^{(\Xi)}_{n-1,n}<1$ and $0<x^{(\Lambda)}_{1,n}<\cdots<x^{(\Lambda)}_{n-1,n}<1$, we study the asymptotic behaviour of the largest zeros $x^{(\Xi)}_{n-1,n}$ and $x^{(\Lambda)}_{n-1,n}$. We prove that the two families have different exponential rates at the right endpoint:
\[
  \frac{1}{n-1}\log\bigl(1-x^{(\Lambda)}_{n-1,n}\bigr)\to-\log4,
  \qquad
  \frac{1}{n-1}\log\bigl(1-x^{(\Xi)}_{n-1,n}\bigr)\to-\log9.
\]
Thus, although the two families share the same global limiting zero distribution, their extreme right zeros approach $1$ on different exponential scales. The proof is based on the representation of $\Xi_n$ and $\Lambda_n$ in terms of Eulerian polynomials of type~B and type~A, respectively, and on an elementary estimate for the smallest negative zero in terms of the first non-constant coefficient.
\end{abstract}

\vspace{0.2cm}

\paragraph{Notation.}
Throughout this manuscript, $\mathbb{N}$ denotes the set of positive integers and $\mathbb{N}_0$ the set of non-negative integers. We write $A_m(x)$ for the Eulerian polynomial of type~A and $B_m(x)$ for the Eulerian polynomial of type~B. The zeros of $\widetilde{\Xi}_n$ and $\widetilde{\Lambda}_n$ in $(0,1)$ are denoted by $x^{(\Xi)}_{k,n}$ and $x^{(\Lambda)}_{k,n}$, respectively, and are ordered as
\[
  0<x^{(\Xi)}_{1,n}<\cdots<x^{(\Xi)}_{n-1,n}<1,
  \qquad
  0<x^{(\Lambda)}_{1,n}<\cdots<x^{(\Lambda)}_{n-1,n}<1.
\]

\vspace{0.5cm}

\section*{Introduction}

In~\cite{TallaWaffo2026arxiv2602.16761}, the polynomials $\Xi_n$ and $\Lambda_n$ were introduced in connection with integral representations for the normalized values $\beta(2n)/\pi^{2n-1}$ and $\zeta(2n+1)/\pi^{2n}$. A central feature of these polynomials is that they admit explicit representations in terms of Eulerian polynomials. More precisely, for $n\ge1$ one has
\begin{equation}\label{eq:Xi-eulerian}
  \Xi_n(t)
  =
  \frac{(-1)^{n+1}}{2^{4n-1}(2n-1)!}
  \frac{(1+t)^{2n-1}}{t}\,
  B_{2n-1}\!\left(-\frac{1-t}{1+t}\right),
\end{equation}
and
\begin{equation}\label{eq:Lambda-eulerian}
  \Lambda_n(t)
  =
  \frac{(-1)^{n+1}}{(2^{2n+1}-1)(2n)!}
  \frac{(1+t)^{2n-1}}{t}\,
  A_{2n}\!\left(-\frac{1-t}{1+t}\right).
\end{equation}
Here $A_m$ and $B_m$ denote the Eulerian polynomials of type~A and type~B, respectively.

Since $\Xi_n$ and $\Lambda_n$ are even polynomials, it is natural to introduce the rescaled families
\begin{equation}\label{eq:tilde-defs}
  \widetilde{\Xi}_n(x):=\Xi_n(\sqrt{x}),
  \qquad
  \widetilde{\Lambda}_n(x):=\Lambda_n(\sqrt{x}).
\end{equation}
The results of~\cite{TallaWaffo2026arxiv2602.16761} imply that these polynomials have degree $n-1$ and that all their zeros are real, simple, and contained in $(0,1)$.

The global distribution of these zeros is governed by the limiting distribution function
\[
  F(x)
  =
  \frac{2}{\pi}
  \arctan\!\left(
    \frac1\pi\log\frac{1+\sqrt{x}}{1-\sqrt{x}}
  \right),
  \qquad 0<x<1.
\]
This limiting law describes the macroscopic location of the zeros. For instance, at the left endpoint one obtains $x_{k,n}\sim(\pi^4/16)k^2/(n-1)^2$ for fixed $k$, so the smallest zeros live on the scale $(n-1)^{-2}$.

The purpose of the present note is to isolate the corresponding phenomenon at the right endpoint. In contrast with the left edge, the largest zeros approach $1$ exponentially fast. Moreover, the two families $\widetilde{\Xi}_n$ and $\widetilde{\Lambda}_n$ do not have the same exponential rate. Our main result is
\[
  \frac{1}{n-1}\log\bigl(1-x^{(\Lambda)}_{n-1,n}\bigr)\to-\log4,
  \qquad
  \frac{1}{n-1}\log\bigl(1-x^{(\Xi)}_{n-1,n}\bigr)\to-\log9.
\]
The proof is elementary once the Eulerian representations \eqref{eq:Xi-eulerian} and \eqref{eq:Lambda-eulerian} are used.

\vspace{0.3cm}

\section{Preliminaries}\label[section]{sec:preliminaries}

We first record the elementary transformation connecting $(0,1)$ with the negative real axis.

\begin{lemma}\label[lemma]{lem:mobius}
Let $0<x<1$ and define $z=-(1-\sqrt{x})/(1+\sqrt{x})$. Writing $z=-a$, with $a\in(0,1)$, one has $x=((1-a)/(1+a))^2$ and therefore
\begin{equation}\label{eq:one-minus-x}
  1-x=\frac{4a}{(1+a)^2}.
\end{equation}
In particular, $1-x\sim4a$ as $a\to0^+$.
\end{lemma}

\begin{proof}
Since $z=-a$, we have $a=(1-\sqrt{x})/(1+\sqrt{x})$. Solving for $\sqrt{x}$ gives $\sqrt{x}=(1-a)/(1+a)$, hence $x=((1-a)/(1+a))^2$. Consequently,
\[
  1-x
  =
  1-\left(\frac{1-a}{1+a}\right)^2
  =
  \frac{(1+a)^2-(1-a)^2}{(1+a)^2}
  =
  \frac{4a}{(1+a)^2},
\]
which proves the claim.
\end{proof}

The following elementary lemma is the key estimate used below.

\begin{lemma}\label[lemma]{lem:first-coeff-root}
Let $P(z)=1+c_1z+c_2z^2+\cdots+c_dz^d$ be a real polynomial with positive coefficients and only real negative zeros. Write its zeros as $-a_1,-a_2,\ldots,-a_d$, where $0<a_1\le a_2\le\cdots\le a_d$. Then
\begin{equation}\label{eq:a1-bounds}
  \frac1{c_1}\le a_1\le \frac{d}{c_1}.
\end{equation}
Consequently, $\log a_1=-\log c_1+O(\log d)$.
\end{lemma}

\begin{proof}
Since the constant term of $P$ equals $1$, we may write $P(z)=\prod_{j=1}^{d}(1+z/a_j)$. Comparing the coefficient of $z$ gives $c_1=\sum_{j=1}^{d}1/a_j$. Since $a_1\le a_j$ for all $j$, we have $1/a_j\le1/a_1$, and thus $c_1\le d/a_1$, which gives $a_1\le d/c_1$. On the other hand, $c_1\ge1/a_1$, hence $a_1\ge1/c_1$. This proves \eqref{eq:a1-bounds}. Taking logarithms gives $-\log c_1\le\log a_1\le-\log c_1+\log d$, which is equivalent to the claimed estimate.
\end{proof}

We shall also need the first non-constant coefficients of the Eulerian polynomials of type~A and type~B.

\begin{lemma}\label[lemma]{lem:first-coefficients}
For the Eulerian polynomials of type~A and type~B, respectively, one has
\begin{align}
  [z]A_m(z)&=2^m-m-1, \label{eq:A-first-coeff}\\
  [z]B_m(z)&=3^m-(m+1). \label{eq:B-first-coeff}
\end{align}
\end{lemma}

\begin{proof}
For type~A, the coefficient $[z]A_m(z)$ is the Eulerian number counting permutations of $\{1,\ldots,m\}$ with exactly one descent. This number is well known to be $\left\langle {m\atop 1}\right\rangle=2^m-m-1$, which gives \eqref{eq:A-first-coeff}.

For type~B, we use the classical identity
\[
  \sum_{k=0}^{\infty}(2k+1)^m z^k
  =
  \frac{B_m(z)}{(1-z)^{m+1}},
  \qquad |z|<1.
\]
Equivalently, $B_m(z)=(1-z)^{m+1}(1+3^mz+5^mz^2+\cdots)$. The coefficient of $z$ on the right-hand side is $3^m-(m+1)$, proving \eqref{eq:B-first-coeff}.
\end{proof}

\vspace{0.3cm}

\section{The right edge rates}\label[section]{sec:right-edge}

We now prove the main result.

\begin{theorem}[Right edge rate]\label{thm:right-edge-rate}
Let $0<x^{(\Xi)}_{1,n}<\cdots<x^{(\Xi)}_{n-1,n}<1$ be the zeros of $\widetilde{\Xi}_n$, and let $0<x^{(\Lambda)}_{1,n}<\cdots<x^{(\Lambda)}_{n-1,n}<1$ be the zeros of $\widetilde{\Lambda}_n$. Then
\begin{align}
  \lim_{n\to\infty}
  \frac{1}{n-1}
  \log\bigl(1-x^{(\Lambda)}_{n-1,n}\bigr)
  &=
  -\log 4, \label{eq:Lambda-right-rate}\\
  \lim_{n\to\infty}
  \frac{1}{n-1}
  \log\bigl(1-x^{(\Xi)}_{n-1,n}\bigr)
  &=
  -\log 9. \label{eq:Xi-right-rate}
\end{align}
\end{theorem}

\begin{proof}
We first treat $\widetilde{\Lambda}_n$. By \eqref{eq:Lambda-eulerian}, the zeros of $\widetilde{\Lambda}_n$ are obtained from the negative zeros of $A_{2n}$ through $z=-(1-\sqrt{x})/(1+\sqrt{x})$. The largest zero of $\widetilde{\Lambda}_n$ corresponds to the negative zero of $A_{2n}$ closest to the origin; write it as $-a^{(A)}_n$, with $a^{(A)}_n>0$. Then
\begin{equation}\label{eq:Lambda-largest-via-a}
  x^{(\Lambda)}_{n-1,n}
  =
  \left(\frac{1-a^{(A)}_n}{1+a^{(A)}_n}\right)^2,
\end{equation}
and, by \cref{lem:mobius},
\begin{equation}\label{eq:Lambda-gap-via-a}
  1-x^{(\Lambda)}_{n-1,n}
  =
  \frac{4a^{(A)}_n}{(1+a^{(A)}_n)^2}.
\end{equation}

Write $A_{2n}(z)=1+c^{(A)}_{n,1}z+\cdots$. By \cref{lem:first-coefficients},
\[
  c^{(A)}_{n,1}
  =
  [z]A_{2n}(z)
  =
  2^{2n}-2n-1
  =
  4^n-2n-1.
\]
Since the degree of $A_{2n}$ is $2n-1$, \cref{lem:first-coeff-root} gives
\[
  \frac{1}{c^{(A)}_{n,1}}
  \le
  a^{(A)}_n
  \le
  \frac{2n-1}{c^{(A)}_{n,1}},
\]
hence $\log a^{(A)}_n=-\log c^{(A)}_{n,1}+O(\log n)$. Since $c^{(A)}_{n,1}=4^n-2n-1$, it follows that $(n-1)^{-1}\log a^{(A)}_n\to-\log4$.

Using \eqref{eq:Lambda-gap-via-a}, we have
\[
  \log\bigl(1-x^{(\Lambda)}_{n-1,n}\bigr)
  =
  \log a^{(A)}_n+\log4-2\log(1+a^{(A)}_n).
\]
Since $a^{(A)}_n\to0$, the last two terms are $o(n)$ after division by $n-1$. Therefore $(n-1)^{-1}\log(1-x^{(\Lambda)}_{n-1,n})\to-\log4$, proving \eqref{eq:Lambda-right-rate}.

We now treat $\widetilde{\Xi}_n$. By \eqref{eq:Xi-eulerian}, the zeros of $\widetilde{\Xi}_n$ are obtained from the negative zeros of $B_{2n-1}$ through the same transformation $z=-(1-\sqrt{x})/(1+\sqrt{x})$. Let $-a^{(B)}_n$, with $a^{(B)}_n>0$, be the zero of $B_{2n-1}$ closest to the origin. Then
\[
  x^{(\Xi)}_{n-1,n}
  =
  \left(\frac{1-a^{(B)}_n}{1+a^{(B)}_n}\right)^2,
  \qquad
  1-x^{(\Xi)}_{n-1,n}
  =
  \frac{4a^{(B)}_n}{(1+a^{(B)}_n)^2}.
\]
Write $B_{2n-1}(z)=1+c^{(B)}_{n,1}z+\cdots$. By \cref{lem:first-coefficients},
\[
  c^{(B)}_{n,1}
  =
  [z]B_{2n-1}(z)
  =
  3^{2n-1}-2n.
\]
Since the degree of $B_{2n-1}$ is $2n-1$, \cref{lem:first-coeff-root} gives
\[
  \frac{1}{c^{(B)}_{n,1}}
  \le
  a^{(B)}_n
  \le
  \frac{2n-1}{c^{(B)}_{n,1}},
\]
hence $\log a^{(B)}_n=-\log c^{(B)}_{n,1}+O(\log n)$. Since $c^{(B)}_{n,1}=3^{2n-1}-2n$, we obtain
\[
  \frac{1}{n-1}\log a^{(B)}_n
  \to
  -2\log3
  =
  -\log9.
\]
Finally,
\[
  \log\bigl(1-x^{(\Xi)}_{n-1,n}\bigr)
  =
  \log a^{(B)}_n+\log4-2\log(1+a^{(B)}_n).
\]
Dividing by $n-1$ and letting $n\to\infty$ yields $(n-1)^{-1}\log(1-x^{(\Xi)}_{n-1,n})\to-\log9$, proving \eqref{eq:Xi-right-rate}.
\end{proof}

\vspace{0.4cm}

\section{Interpretation}\label[section]{sec:interpretation}

The result above shows that the two polynomial families behave differently at the extreme right edge. Although the empirical zero measures of $\widetilde{\Xi}_n$ and $\widetilde{\Lambda}_n$ have the same limiting distribution, the largest zeros have distinct exponential scales.

For $\widetilde{\Lambda}_n$, \cref{thm:right-edge-rate} gives
\[
  1-x^{(\Lambda)}_{n-1,n}
  =
  \exp\bigl(-(\log4+o(1))(n-1)\bigr),
\]
so the rightmost zero of $\widetilde{\Lambda}_n$ approaches $1$ on the scale $4^{-n}$. For $\widetilde{\Xi}_n$, one obtains instead
\[
  1-x^{(\Xi)}_{n-1,n}
  =
  \exp\bigl(-(\log9+o(1))(n-1)\bigr),
\]
so the rightmost zero approaches $1$ on the faster scale $9^{-n}$.

This phenomenon is not visible from the global limiting distribution alone. The limiting distribution captures the macroscopic distribution of zeros in $(0,1)$, but the largest zero is governed by the behaviour of the Eulerian zero closest to the origin. This zero, in turn, is controlled on the exponential scale by the first non-constant coefficient of the corresponding Eulerian polynomial. In short,
\[
  [z]A_{2n}(z)\sim4^n
  \quad\text{and}\quad
  [z]B_{2n-1}(z)\sim9^n
\]
imply the two different right-edge scales $4^{-n}$ and $9^{-n}$, respectively.

\vspace{0.4cm}

\section{Numerical illustration}\label[section]{sec:numerics}

We close with a short numerical illustration. The following table lists the normalized logarithmic gaps
\[
  R^{(\Xi)}_n
  :=
  \frac{1}{n-1}\log\bigl(1-x^{(\Xi)}_{n-1,n}\bigr),
  \qquad
  R^{(\Lambda)}_n
  :=
  \frac{1}{n-1}\log\bigl(1-x^{(\Lambda)}_{n-1,n}\bigr).
\]
The limiting values are $-\log9\approx-2.197224577$ and $-\log4\approx-1.386294361$.

\begin{center}
\begin{tabular}{c|c|c}
$n$ & $R^{(\Xi)}_n$ & $R^{(\Lambda)}_n$ \\
\hline
$10$ & $-2.165$ & $-1.386$ \\
$15$ & $-2.177$ & $-1.386$ \\
$20$ & $-2.182$ & $-1.386$ \\
$25$ & $-2.185$ & $-1.386$ \\
$30$ & $-2.187$ & $-1.386$
\end{tabular}
\end{center}

The numerical data are consistent with \cref{thm:right-edge-rate}. The convergence in the type~A case is already very sharp for moderate $n$, whereas the type~B sequence approaches its limit more slowly but clearly tends toward $-\log9$.

\vspace{0.4cm}

\section*{Acknowledgments}

The author acknowledges the use of an AI language model for assistance with the presentation of the manuscript, numerical experimentation, and verification of elementary asymptotic estimates.

\printbibliography

\end{document}